\documentclass[11pt]{amsart}

\usepackage[english]{babel}
\usepackage[utf8]{inputenc}

\usepackage{amsmath}
\usepackage{amssymb}
\usepackage{amsfonts}
\usepackage{amscd}
\usepackage{amsthm}
\usepackage{mathtools}
\usepackage[enableskew]{youngtab}

\usepackage{float}
\usepackage{caption}
\usepackage[all]{xy}

\usepackage[enableskew]{youngtab}
\usepackage{ytableau}

\usepackage{tikz}
\definecolor{light-gray}{gray}{0.7}

\makeatletter
\def\@settitle{\begin{center}%
    \baselineskip14\p@\relax
    \bfseries
    \@title
  \end{center}%
}
\makeatother

\newcommand{\nc}{\newcommand}

\nc{\exto}[1]{\stackrel{#1}{\longrightarrow}}
\nc{\dlim}{{\mathop{\lim\limits_{\longrightarrow}\,}}}
\nc{\ilim}{{\mathop{\lim\limits_{\longleftarrow}\,}}}
\nc{\hocolim}{{\mathop{\sf hocolim}\,}}
\nc{\holim}{{\mathop{\sf holim}}}
\nc{\lan}{\big\langle}
\nc{\ran}{\big\rangle}

\nc{\kk}{{\mathsf{k}}}

\nc{\HH}{{\mathbf{H}}}
\nc{\DD}{{\mathbb{D}}}
\nc{\LL}{{\mathbb{L}}}
\nc{\PP}{{\mathbb{P}}}
\nc{\QQ}{{\mathbb{Q}}}
\nc{\RR}{{\mathbb{R}}}
\nc{\ZZ}{{\mathbb{Z}}}

\nc{\CA}{{\mathcal{A}}}
\nc{\CB}{{\mathcal{B}}}
\nc{\CC}{{\mathcal{C}}}
\nc{\D}{{\mathcal{D}}}
\nc{\CE}{{\mathcal{E}}}
\nc{\CF}{{\mathcal{F}}}
\nc{\CG}{{\mathcal{G}}}
\nc{\CH}{{\mathcal{H}}}
\nc{\CL}{{\mathcal{L}}}
\nc{\CM}{{\mathcal{M}}}
\nc{\CN}{{\mathcal{N}}}
\nc{\CO}{{\mathcal{O}}}
\nc{\CQ}{{\mathcal{Q}}}
\nc{\CR}{{\mathcal{R}}}
\nc{\CS}{{\mathcal{S}}}
\nc{\CT}{{\mathcal{T}}}
\nc{\CU}{{\mathcal{U}}}
\nc{\CV}{{\mathcal{V}}}
\nc{\CW}{{\mathcal{W}}}
\nc{\CX}{{\mathcal{X}}}
\nc{\CY}{{\mathcal{Y}}}
\nc{\CMo}{{\mathcal{M}^\circ}}
\nc{\Co}{{{C}^\circ}}

\nc{\BY}{{\overline{Y}}}
\nc{\BYD}{{\overline{Y}{}^{|D|}}}
\nc{\OZ}{{\overline{Z}}}
\nc{\bg}{{\bar{g}}}

\nc{\bq}{{\mathbf{q}}}
\nc{\BD}{{\mathbf{D}}}
\nc{\BG}{{\mathbf{G}}}
\nc{\BM}{{\mathbf{M}}}
\nc{\BP}{{\mathbf{P}}}
\nc{\BZ}{{\mathbf{Z}}}
\nc{\BPr}{{\mathsf{P}}}
\nc{\BL}{{\mathbf{L}}}
\nc{\BR}{{\mathbf{R}}}
\nc{\BRO}[1]{{{\mathbf{R}}^{\circ}_{#1}}}
\nc{\BRD}[1]{{{\mathbf{R}}^{|D|}_{#1}}}
\nc{\BRP}[1]{{{\mathbf{R}}^{1}_{#1}}}
\nc{\BRTP}[1]{{{\mathbf{\tilde{R}}}{}^{1}_{#1}}}
\nc{\BS}{{\mathbf{S}}}
\nc{\BMS}{{{\mathbf{M}}^{{s}}}}
\nc{\BMSS}{{{\mathbf{M}}^{{ss}}}}
\nc{\BMZ}{{\mathbf{M}^{\circ}}}
\nc{\BCL}{{\mathbf{L}}}

\nc{\PCC}{{{}^\perp\CC}}

\nc{\Cl}{{\mathsf{Cliff}}}
\nc{\Clev}{{\mathop{\mathsf{Cliff}}^{\circ}}}

\nc{\FA}{{\mathfrak{A}}}
\nc{\FB}{{\mathfrak{B}}}

\nc{\fa}{{\mathfrak{a}}}
\nc{\fb}{{\mathfrak{b}}}
\nc{\fg}{{\mathfrak{g}}}
\nc{\fn}{{\mathfrak{n}}}
\nc{\fp}{{\mathfrak{p}}}
\nc{\FD}{{\mathfrak{D}}}
\nc{\FE}{{\mathfrak{E}}}
\nc{\FL}{{\mathfrak{L}}}
\nc{\FM}{{\mathfrak{M}}}
\nc{\FS}{{\mathsf{S}}}

\nc{\sfc}{{\mathsf{c}}}
\nc{\sfch}{{\mathsf{ch}}}
\nc{\sfh}{{\mathsf{h}}}

\nc{\SK}{{\mathsf{K}}}
\nc{\SM}{{\mathsf{M}}}
\nc{\SO}{{\mathsf{O}}}
\nc{\SQ}{{\mathsf{Q}}}
\nc{\SPV}{{\mathsf{S}^+\mathsf{V}}}
\nc{\SMV}{{\mathsf{S}^-\mathsf{V}}}
\nc{\SPMV}{{\mathsf{S}^\pm\mathsf{V}}}
\nc{\SX}{{S_X}}
\nc{\SY}{{S_Y}}
\nc{\phipsi}{{q}}
\nc{\eps}{\varepsilon}

\nc{\pim}{{\pi_-}}
\nc{\pip}{{\pi_+}}

\nc{\BE}{{\overline{\CE}}}
\nc{\TE}{{\tilde{\CE}}}
\nc{\TQ}{{\tilde{Q}}}
\nc{\TCF}{{\tilde{\CF}}}
\nc{\TCG}{{\tilde{\CG}}}
\nc{\TCL}{{\tilde{\CL}}}
\nc{\TF}{{\tilde{F}}}
\nc{\TW}{{\tilde{W}}}
\nc{\TCC}{{\tilde{\CC}}}
\nc{\TCX}{{\tilde{\CX}}}
\nc{\TCY}{{\tilde{\CY}}}
\nc{\TPhi}{{\tilde{\Phi}}}
\nc{\OPhi}{{\bar{\Phi}}}
\nc{\txi}{{\tilde{\xi}}}
\nc{\tp}{{\tilde{p}}}
\nc{\tq}{{\tilde{q}}}
\nc{\tzeta}{{\tilde{\zeta}}}
\nc{\tpi}{{\tilde{\pi}}}

\nc{\halpha}{{\hat{\alpha}}}
\nc{\HCA}{{\hat{\CA}}}
\nc{\HCB}{{\hat{\CB}}}
\nc{\HCC}{{\hat{\CC}}}
\nc{\HE}{{\widehat{\CE}}}
\nc{\HX}{{\hat{X}}}
\nc{\hxi}{{\hat{\xi}}}

\nc{\UH}{{\mathcal{H}}}

\nc{\TM}{{\widetilde{M}}}
\nc{\TCM}{{\widetilde{\CM}}}
\nc{\TU}{{\widetilde{U}}}
\nc{\TX}{{\widetilde{X}}}
\nc{\TY}{{\widetilde{Y}}}
\nc{\TYO}{{{\widetilde{Y}}^\circ}}
\nc{\barf}{{\bar{f}}}
\nc{\te}{{\tilde{e}}{}}
\nc{\tf}{{\tilde{f}}}
\nc{\ti}{{\tilde{\imath}}}
\nc{\tj}{{\tilde{\jmath}}}
\nc{\ty}{{\tilde{y}}}
\nc{\tphi}{{\tilde{\phi}}}

\nc{\urho}{{\underline{\rho}}}

\nc{\LRA}{\Leftrightarrow}
\nc{\RA}{\Rightarrow}
\nc{\lotimes}{\mathbin{\mathop{\otimes}\limits^{\mathbb{L}}}}
\nc{\CEnd}{\mathop{\mathcal{E}\mathit{nd}}\nolimits}
\nc{\CExt}{\mathop{\mathcal{E}\mathit{xt}}\nolimits}
\nc{\CHom}{\mathop{\mathcal{H}\mathit{om}}\nolimits}
\nc{\RH}{\mathop{{\mathsf{R}}\Gamma}\nolimits}
\nc{\RGamma}{\mathop{{\mathsf{R}}\Gamma}\nolimits}
\nc{\RHom}{\mathop{\mathsf{RHom}}\nolimits}
\nc{\RCHom}{\mathop{\mathsf{R}\mathcal{H}\mathit{om}}\nolimits}
\nc{\RG}{\mathop{\mathsf{R\Gamma}}\nolimits}
\nc{\Hom}{\mathop{\mathsf{Hom}}\nolimits}
\nc{\Ext}{\mathop{\mathsf{Ext}}\nolimits}
\nc{\End}{\mathop{\mathsf{End}}\nolimits}
\nc{\Tor}{\mathop{\mathsf{Tor}}\nolimits}
\nc{\Tordim}{\mathop{\mathsf{Tor}\text{\rm-}\mathsf{dim}}\nolimits}
\nc{\Hilb}{\mathop{\mathsf{Hilb}}\nolimits}
\nc{\Spec}{\mathop{\mathsf{Spec}}\nolimits}
\nc{\Pic}{\mathop{\mathsf{Pic}}\nolimits}

\nc{\Tr}{\mathop{\mathsf{Tr}}\nolimits}
\nc{\Cone}{\mathop{\mathsf{Cone}}\nolimits}
\nc{\Fiber}{\mathop{\mathsf{Fiber}}\nolimits}
\nc{\Ker}{\mathop{\mathsf{Ker}}\nolimits}
\nc{\Coker}{\mathop{\mathsf{Coker}}\nolimits}
\nc{\codim}{\mathop{\mathsf{codim}}\nolimits}
\nc{\sing}{{\mathsf{sing}}}
\nc{\supp}{\mathop{\mathsf{supp}}}
\nc{\vol}{\mathop{\mathsf{vol}}\nolimits}
\nc{\perf}{{\mathsf{perf}}}
\nc{\rank}{\mathop{\mathsf{rank}}}
\nc{\Pf}{{\mathsf{Pf}}}
\nc{\Gr}{{\mathsf{Gr}}}
\nc{\OGr}{{\mathsf{OGr}}}
\nc{\Flag}{{\mathsf{Fl}}}
\nc{\Kosz}{{\mathsf{Kosz}}}
\nc{\LGr}{{\mathsf{LGr}}}
\nc{\GTGr}{{\mathsf{G_2Gr}}}
\nc{\GTF}{{\mathsf{G_2F}}}
\nc{\OF}{{\mathsf{OF}}}
\nc{\Fl}{{\mathsf{Fl}}}
\nc{\Bl}{{\mathsf{Bl}}}
\nc{\GL}{{\mathsf{GL}}}
\nc{\PGL}{{\mathsf{PGL}}}
\nc{\SL}{{\mathsf{SL}}}
\nc{\SP}{{\mathsf{Sp}}}
\nc{\Spin}{{\mathsf{Spin}}}
\nc{\Tot}{{\mathsf{Tot}}}
\nc{\ev}{{\mathsf{ev}}}
\nc{\od}{{\mathsf{odd}}}
\nc{\coev}{{\mathsf{coev}}}
\nc{\id}{{\mathsf{id}}}
\nc{\opp}{{\mathsf{opp}}}
\nc{\PS}{{{\PP^3}}}
\nc{\Qu}{{{Q^3}}}
\nc{\tdim}{\mathop{\Tor\dim}}
\nc{\ecart}{{\fbox{$\scriptstyle\mathsf{EC}$}}}
\nc{\ad}{{\mathop{\mathsf ad}}}
\nc{\sg}{{\mathop{\mathsf sg}}}
\nc{\hf}{{\mathop{\mathsf hf}}}
\nc{\gr}{{\mathop{\mathsf gr}}}
\nc{\qgr}{{\mathop{\mathsf qgr}}}
\nc{\Coh}{{\mathop{\mathsf Coh}}}
\nc{\Ab}{{\mathop{\mathcal{A}\mathit{b}}}}
\nc{\Ccoh}{{\mathop{\mathsf Ccoh}}}
\nc{\Qcoh}{{\mathop{\mathsf Qcoh}}}
\nc{\At}{{\mathop{\mathsf{At}}\nolimits}}
\nc{\tra}{{\mathsf{T}}}
\nc{\fsl}{{\mathfrak{sl}}}
\nc{\fso}{{\mathfrak{so}}}
\nc{\fgl}{{\mathfrak{gl}}}

\nc{\AAV}{{\mathcal{AAV}}}

\nc{\Rep}{{\mathsf{Rep}}}

\nc{\git}{{/\!\!/\!{}_\chi}}

\nc{\HOH}{{\mathsf H\mathsf H}}
\nc{\HHE}{{\mathsf H\mathsf E}}

\def\presuper#1#2%
  {\mathop{}%
   \mathopen{\vphantom{#2}}^{#1}%
   \kern-\scriptspace%
   #2}

\nc{\You}{{\mathsf{Y}}}
\nc{\lort}[1]{\vphantom{#1}^\perp{}#1}
\nc{\trsp}[1]{\presuper{t}{#1}}

\nc{\dcut}[2]{#1^{(#2)}}
\nc{\ncut}[2]{b_{#1}^{(#2)}}

\nc{\sG}{\mathsf{G}}
\nc{\sP}{\mathsf{P}}
\nc{\sU}{\mathsf{U}}
\nc{\sL}{\mathsf{L}}
\nc{\CK}{\mathcal{K}}

\DeclarePairedDelimiter\ceil{\lceil}{\rceil}
\DeclarePairedDelimiter\floor{\lfloor}{\rfloor}

\theoremstyle{plain}

\newtheorem{theorem}{Theorem}[section]
\newtheorem{conjecture}[theorem]{Conjecture}
\newtheorem{lemma}[theorem]{Lemma}
\newtheorem{proposition}[theorem]{Proposition}
\newtheorem{corollary}[theorem]{Corollary}

\theoremstyle{definition}

\newtheorem{definition}[theorem]{Definition}

\theoremstyle{remark}

\newtheorem{remark}[theorem]{Remark}
\newtheorem{example}[theorem]{Example}

\title{On a conjecture of Kuznetsov and Polishchuk}
\author{Anton Fonarev}
\address{\sloppy
\parbox{0.9\textwidth}{
		Laboratory of Algebraic Geometry, SU-HSE,
		7 Vavilova Str., Moscow 117312 Russia
\hfill
}\bigskip}
\thanks{This work is supported by the RSF under a grant 14-50-00005.}
\email{avfonarev@mi.ras.ru}
\date{}

\begin{document}

\begin{abstract}
  We prove a conjecture by A.~Kuznetsov and A.~Polishchuk on the existence of some particular
  full exceptional collections in bounded derived categories of coherent sheaves on
  Grassmannian varieties.
\end{abstract}

\maketitle

\section{Introduction}
The study of derived categories of coherent sheaves on algebraic varieties is
one of the main branches of contemporary algebraic geometry. In particular, one would
often like to construct full exceptional collections that in are possession enjoyable
properties. These questions go back to seminal papers by A.~Beilinson~\cite{Bei}
and M.~Kapranov~\cite{Kap}, where full strong exceptional collections were constructed
in bounded derived categories of projective spaces and general Grassmannians
respectively.

Since Kapranov's result, a~long standing conjecture predicts that the bounded derived
category of coherent sheaves on a~rational homogeneous variety admits a full exceptional collection.
In their recent work~\cite{KuzPol} A.~Kuznetsov and A.~Polishchuk proposed a general
method of constructing exceptional collections on homogeneous spaces and used it
to find exceptional collections of maximal length in bounded derived categories of
generalized Grassmannians of type $BCD$. Conjectural fullness of these collections remains
an important open question.

The present paper concentrates on proving a conjecture stated in~\cite{KuzPol}.
In particular, the authors predicted existence of a whole set of full exceptional
collections of a particular form on classical Grassmannian (of type $A$). This is
an interesting question in its own right as it gives a particular way to incrementally mutate
Kapranov's collection into its dual. Besides that, we believe that the methods
developed along the may turn out to be useful for the study of Kuznetsov and Polishchuk's construction.
Let us now state the conjecture in question.

Throughout the paper we fix a pair of integers $0<k<n$. The main object of study is the Grassmannian
$X=\Gr(k,V)$ of $k$-dimensional subspaces in a fixed $n$-dimensional vector space $V$ over
an~algebraically closed field of characteristic zero.

Consider a strongly increasing continuous function $\Gamma:[0,n-k]\to\RR$ with boundary conditions
$\Gamma(0)=0$ and $\Gamma(n-k)=k$. We prefer to think about such functions as monotonous paths
going from the lower left corner of some rectangle of width $n-k$ and height $k$ to the upper
right one. Let $p_i=(x_i,y_i)$, $i=0,\ldots,l(\Gamma)$ be the points on the graph of $\Gamma$
with at least one integer coordinate ordered from left to right. In particular, $p_0=(0,0)$ and $p_{l(\Gamma)}=(n-k,k)$.

Denote by $\Pi$ the rectangle $[0,n-k]\times[0,k]\subset\RR^2$.
For every point $p=(x,y)\in\Pi$ define the
\emph{block} $\Bl_p=\You_{n-k-x,y}\times\You_{k-y,x}$, where $\You_{w,h}$ is the set of all
Young diagrams inscribed in a rectangle of width $\floor{w}$ and height $\floor{h}$.
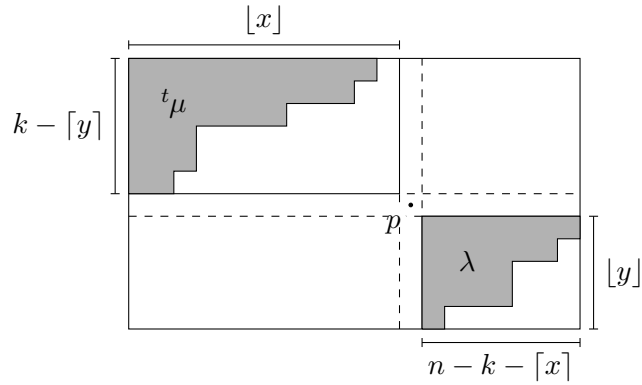
\begin{center}
\begin{tikzpicture}[scale=0.6]
    \draw (0,0) rectangle (10,6);
    \draw [dashed] (0,2.5) -- (6.5,2.5) -- (6.5,6);
    \draw [dashed] (6,0) -- (6,3) -- (10,3);
    \draw [fill=black] (6.25,2.75) circle [radius=0.04];
    \fill [white] (5.95,2.45) circle [radius=0.2];
    \node [below left] at (6.25,2.75) {$p$};
    \draw (0,3) -- (6,3) -- (6,6);
    \draw (6.5,0) -- (6.5,2.5) -- (10,2.5);
    
    \draw (-0.3,3) -- (-0.3,6);
    \draw (-0.4,3) -- (-0.2,3);
    \draw (-0.4,6) -- (-0.2,6);
    \node [left] at (-0.3,4.5) {$k-\ceil{y}$};
    
    \draw (0,6.3) -- (6,6.3);
    \draw (0,6.2) -- (0,6.4);
    \draw (6,6.2) -- (6,6.4);
    \node [above] at (3,6.3) {$\floor{x}$};

    \draw (10.3,0) -- (10.3,2.5);
    \draw (10.4,2.5) -- (10.2,2.5);
    \draw (10.4,0) -- (10.2,0);
    \node [right] at (10.3,1.25) {$\floor{y}$};

    \draw (6.5,-0.3) -- (10,-0.3);
    \draw (6.5,-0.2) -- (6.5,-0.4);
    \draw (10,-0.2) -- (10,-0.4);
    \node [below] at (8.25,-0.3) {$n-k-\ceil{x}$};
    
    \draw [fill=light-gray] (0,3) -- (0,6) -- (5.5,6) -- (5.5,5.5)
    -- (5,5.5) -- (5,5) -- (3.5,5) -- (3.5,4.5) -- (1.5,4.5) --
    (1.5,3.5) -- (1,3.5) -- (1,3) -- (0,3);
    \node at (1,5) {$\trsp{\mu}$};
    
    \draw [fill=light-gray] (6.5,0) -- (6.5,2.5) -- (10,2.5) -- (10,2)
    -- (9.5,2) -- (9.5,1.5) -- (8.5,1.5) -- (8.5,0.5) -- (7,0.5)
    -- (7,0) -- (6.5,0);
    \node at (7.5,1.5) {$\lambda$};
\end{tikzpicture}
\end{center}
To every block we associate the subcategory
\[
    \CB_p=\left<\Sigma^\lambda\CU^*\otimes \Sigma^\mu(V/\CU)\ \mid\ (\lambda,\mu)\in\Bl_p \right> \subset D^b(X),
\]
where $D^b(X)$ is the bounded derived category of coherent sheaves on $X$,
$\CU$ is the tautological subbundle of rank~$k$, and $\Sigma^\alpha$ is the Schur functor
corresponding to the diagram $\alpha$.

\begin{conjecture}[{\cite[Conjecture~9.8]{KuzPol}}]\label{conj}
    For every path $\Gamma$ there is a semi-orthogonal decomposition
    \begin{equation*}
        D^b(X)=\left<\CB_{p_0},\CB_{p_1}(1),\ldots,\CB_{p_{l(\Gamma)}}(l(\Gamma))\right>,
    \end{equation*}
    and every component $\CB_{p_i}(i)$ is generated by a full exceptional collection.
\end{conjecture}

The paper is organized as follows. In the second section we recollect all the preliminary
statements and notation. We study some exceptional equivariant vector bundles on Grassmannians
in the third section and show that every semi-orthogonal component $\CB_p$ is generated by a full
exceptional collection consisting of equivariant vector bundles of that form.
A class of exact sequences, called \emph{staircase complexes,} is constructed in the fourth section.
Finally, in the fifth section we use staircase complexes to give a full proof of Conjecture~\ref{conj}.

The main results of this paper were announced in~\cite{Fon2}.

\subsection*{Acknowledgements}
I would like to thank A.~Kuznetsov for his eternal patience and delicate supervision,
S.~Belcher for correcting a total of fifty six English language mistakes and misprints in the draft,
and the Max Planck Institute for Mathematics in Bonn for hospitality.

\section{Preliminaries}

\subsection{Exceptional collections and semi-orthogonal decompositions}
Let $\CT$ be a linear triangulated category over an algebraically closed field $\kk$ of characteristic zero.

\begin{definition}
    An ordered collection $\CA_1,\CA_2,\ldots,\CA_t\subseteq\CT$ of full triangulated subcategories
    \emph{semi-orthogonal decomposition,} \mbox{if} the following two properties hold:
    \begin{enumerate}
        \item $\Hom_\CT(F,E) = 0$ for any objects $E\in\CA_i$, $F\in\CA_j$, such that $1\leq i<j\leq t$;
        \item the smallest full triangulated subcategory containing all $\CA_i$ coincides with $\CT$.
    \end{enumerate}
    A semi-orthogonal decomposition with \emph{components} $\CA_i$ is denoted by $\CT=\left<\CA_1,\CA_2,\ldots,\CA_t\right>$.
\end{definition}

\begin{definition}\label{def:eo}
    An object $E\in\CT$ is called \emph{exceptional,} if
    \[
        \Hom(E,E[t]) = \begin{cases}
            \kk, & t=0,\\
            0,   & t\neq 0.
        \end{cases}
    \]
\end{definition}

\begin{definition}\label{def:ec}
    A sequence of exceptional objects $E_1,E_2,\ldots,E_t$ is called an
    \emph{exceptional collection,} if for all $1\leq i<j\leq t$ and all $s\in\ZZ$ one has
    $\Hom(E_j,E_i[s])=0$. An exceptional collection is \emph{full,} if $\CT=\left<E_1,E_2,\ldots,E_t\right>$.
\end{definition}

There is a braid group action on the set of full exceptional collections (by means of so called mutations).
Particularly interesting mutations include dual collections, that we characterize by a cohomological
criterion. Recall that in a triangulated category $\CT$ one defines $\Ext^k(X,Y)=\Hom(X,Y[k])$.

\begin{proposition}[Gorodentsev,~\cite{Rudakov}]
    Let $\CT=\left<E_1,E_2,\ldots,E_t\right>$ be a full exceptional collection.
    A collection of objects $(F_{t}, F_{t-1},\ldots,F_1)$, such that
    \[
        \Ext^\bullet(F_i,E_j) = \begin{cases}
            \kk, & i=j\\
            0,   & i\neq j
        \end{cases}
    \]
    also forms a full exception collection, called the \emph{left dual collection.} Interchanging
    $F_i$ and $E_j$ in the cohomological condition we obtain the \emph{left dual collection.}
\end{proposition}

\subsection{Young diagrams}
One usually identifies Young diagrams with non-increasing positive integer sequences of finite length.
However, it will be more convenient for us to think of them as dominant weights of general linear
groups. The latter correspond to arbitrary non-increasing integer sequences of finite length. Thus,
a Young diagram throughout this paper may have rows of zero or negative length. We suggest one thinks
of those rows as consisting of boxes that are placed to the left of a fixed vertical axis, depicting
zero.
 
Let $\You_{w,h}$ denote the set of Young diagrams inscribed in a rectangle of width $w$ and height $h$.
One can identify $\You_{w,h}$ with the set of non-increasing integral sequences $\lambda=(\lambda_1,\lambda_2,\ldots,\lambda_h)$,
such that $w\geq\lambda_1\geq\lambda_2\geq\ldots\geq\lambda_h\geq 0$.
One can also think of these diagrams as of integral paths going from the lower left corner of the
rectangle to the upper right one, moving only upwards or to the right. Such paths consist of
$w$~horizontal and $h$~vertical unit segments. The latter description provides a bijection between
$\You_{w,h}$ and the set of binary sequences of length $w+h$, containing exactly $h$~``ones''.

The group $\ZZ/(w+h)\ZZ$ naturally acts on the set of binary sequences of length $w+h$ by
cyclic shifts, preserving the number of zeros and ones:
\[
    g:a_1a_2\ldots a_n\ \mapsto\ a_na_1a_2\ldots a_{n-1},
\]
where $a_i\in\{0,1\}$ and $g\in \ZZ/(w+h)\ZZ$ is a generator.

The induced action on $\You_{w,h}$ will be called \emph{cyclic,} the image of $\lambda\in\You_{w,h}$
under the generator $g$ denoted by $\lambda\{1\}$ and called the \emph{(cyclic) shift} of $\lambda$.
In terms of Young diagrams $\lambda\{1\}$ can be defined as follows:
\[
\lambda\{1\} = \begin{cases}
    (\lambda_1+1,\lambda_2+1,\ldots,\lambda_h+1), & \text{if } \lambda_1<w, \\
    (\lambda_2,\lambda_3,\ldots,\lambda_h,0), & \text{if } \lambda_1=w.
\end{cases}
\]
The image of $\lambda$ under the action of $g^d$ will be denoted by $\lambda\{d\}$.

Additionally, the group $\ZZ$ acts naturally on the set of all Young diagrams with $h$ rows
by the rule
\[
    \lambda(t)=(\lambda_1+t,\lambda_2+t,\ldots,\lambda_h+t).
\]
We call the diagram $\lambda(t)$ the \emph{twist} of $\lambda$ by $t$.

Let $\lambda\in\You_{w,h}$ be a diagram of maximal width, i.e. $\lambda_1=w$.
For all $i=1,\ldots, w$ define $\dcut{\lambda}{i}$ according to the rule
\[
    \begin{array}{lcl}
    \trsp{\dcut{\lambda}{1}} & = & (\trsp{\lambda}_1, \trsp{\lambda}_2, \ldots, \trsp{\lambda}_{w_\lambda-1}, 0), \\
    \trsp{\dcut{\lambda}{i}} & = & (\trsp{\lambda}_1, \ldots, \trsp{\lambda}_{w_\lambda-i}, \trsp{\lambda}_{w_\lambda-i+2}-1,\ldots, \trsp{\lambda}_{w_\lambda}-1, 0), \\
    \trsp{\dcut{\lambda}{w_\lambda}} & = & (\trsp{\lambda}_2-1, \trsp{\lambda}_3-1,\ldots, \trsp{\lambda}_{w_\lambda}-1, 0). \\
    \end{array}
\]
One gets a sequence of strict inclusions $\lambda\supset\dcut{\lambda}{1}\supset\dcut{\lambda}{2}\supset\ldots\supset\dcut{\lambda}{w}$.
Also, define the numbers $\ncut{\lambda}{i} = |\lambda/\dcut{\lambda}{i}|$, where
$|\alpha/\beta|=|\alpha|-|\beta|$ denotes the number of boxes in a skew diagram.
The diagrams $\dcut{\lambda}{i}$ will be called \emph{band cuts} of $\lambda$.

Band cuts are easy to describe in terms of binary sequences as well. Let $a_1a_2\ldots a_{w+h}$ be
the binary sequence corresponding to $\lambda$. Let $z_1>z_2>\ldots>z_w$ be the indices
of zero terms (recall that there are exactly $w$ of those). The condition $\lambda_1=w$
is equivalent to $a_{w+h}=1$. Now, the diagram $\dcut{\lambda}{i}$ corresponds to the binary sequence
\[
    a_1 a_2\ldots a_{z_i-1} 1 a_{z_i+1}\ldots a_{w+h-1}0.
\]

Informally speaking, in order to find $\dcut{\lambda}{i}$, one should take the inner band of unit width
along the border of $\lambda$ (this band is nothing but the skew diagrams $\lambda/\dcut{\lambda}{w}$)
and cut $i$ rightmost columns of the band out of $\lambda$. The number of boxes one removes
coincides with $\ncut{\lambda}{i}$.

\begin{example}
    Let $\lambda=(4, 4, 2)\in\You_{4,3}$. The diagram $\dcut{\lambda}{i}$ can be obtained by cutting
    out boxes with labels that are less than or equal to $i$, on the following picture.
    \begin{center}
        \ytableausetup{centertableaux}
        \ytableaushort {\none \none \none 1,\none 321, 43} * {4,4,2}
    \end{center}
\end{example}

\subsection{Schur functors}
Let $\lambda$ be a Young diagram with $k$ rows, and let
$\CU$ be a vector bundle of rank $k$ on an algebraic variety $X$. Let $P$ denote the principal
$\GL(k)$-bundle associated to $\CU$.
Consider the irreducible $\GL(k)$-representation $V^{\lambda}$ with highest weight $\lambda$
Then $\Sigma^\lambda\CU=V^\lambda\times_{\GL(k)} P$.
Throughout the paper we adopt a shorthand notation $\CU^\lambda$ for $\Sigma^\lambda\CU$.
By definition $\CU^{-\lambda}=\left(\CU^*\right)^\lambda$.
Analogously, $\CU^{\lambda/\mu}$ will denote the result of applying the skew Schur functor to $\CU$.
One can learn about skew Schur functors from~\cite{Weyman_Art}.

The Littlewood--Richardson rule allows decomposition of tensor product of vector bundles $\CU^\alpha$ and
$\CU^\beta$ into a direct sum with summands of the form $\CU^\gamma$, in other words, compute
multiplicities $m^{\gamma}_{\alpha,\beta}$ in the decomposition
\[
    \CU^\alpha\otimes\CU^\beta = \bigoplus_{\gamma}m^{\gamma}_{\alpha,\beta}\,\CU^\gamma,
\]
where $m^{\gamma}_{\alpha,\beta}\,\CU^\gamma$ denotes the direct sum of $m^{\gamma}_{\alpha,\beta}$
copies of $\CU^\gamma$.

Littlewood--Richardson coefficient also appear when one tries to express skew Schur functors in
terms of usual ones.

\begin{proposition}[{\cite[Proposition~2.3.6]{Weyman_Book}}]\label{pr:skewsh}
    Let $\alpha\subset\beta$ be a pair of plain old Young diagrams. Then
    \[
        \CU^{\beta/\alpha}=\bigoplus_{\gamma}m^{\beta}_{\alpha,\gamma}\,\CU^\gamma.
    \]
\end{proposition}

The following statement is another very useful tool when working with Schur functors.

\begin{proposition}[{\cite[Proposition~2.3.1]{Weyman_Book}}]\label{pr:schses}
    Let $\lambda$ be a Young diagram and let
    \[
        0\to \CU\to \CV\to \CF\to 0
    \]
    be a short exact sequence of vector bundles. There exists a filtration on $\CV^\lambda$
    with associated graded factors isomorphic to
    \[
        \CU^\alpha\otimes\CF^{\lambda/\alpha},
    \]
    where $\alpha$ runs over the set of subdiagrams $\alpha\subseteq\lambda$.
\end{proposition}

We are going to use repeatedly the following statement that follows immediately from the
Littlewood--Richardson rule.
\begin{lemma}\label{lm:schprod}
    Let $\lambda$ and $\mu$ be dominant weights of $\GL(k)$. Let $\CU$ be a vector bundle of rank $k$.
    \begin{enumerate}
        \item For any irreducible summand $\CU^\alpha\subset\CU^\lambda\otimes\CU^\mu$ one has
        \[
            \lambda_i+\mu_k\leq\alpha_i\leq\lambda_1+\mu_i \quad\text{for all $1\leq i\leq k$.}
        \]
        \item If, moreover, $\lambda_k,\mu_k\geq 0$, then for every irreducible summand
        $\CU^\alpha\subset\CU^\lambda\otimes\CU^{-\mu}$ with $\alpha_k\geq 0$
        one has $\alpha\subseteq\lambda$.
    \end{enumerate}
\end{lemma}

\subsection{Vector bundles on Grassmannians}
Let $V$ be an $n$-dimensional vector space.
The Grassmannian $X=\Gr(k,V)$ is a homogeneous space with respect to the group $\sG=\GL(V)$.
Namely, $X=\sG/\sP$, where $\sP$ is a point stabilizer. The subgroup $\sP$ is parabollic.

As the variety $X$ is homogeneous, there is an equivalence of the categories $Coh^\sG(X)$ of coherent 
$\sG$-equivariant sheaves on $X$ and finite-dimensional representations $Rep-\sP$.
Moreover, this is an equivalence of abelian tensor categories: it preserves tensor products, sends
dual objects to dual objects, etc. Thus, the study of equivariant coherent sheaves on $X$ is equivalent
to the study of finite-dimensional representations of $\sP$.

Unfortunately however the group $\sP$ is not reductive. Denote by $\sU\subset\sP$ the unipotent radical $\sP$, and
by $\sL=\sP/\sU$ the Levi quotient. Recall that the projection $\sP\to\sL$ admits a section, which allows
one to make $\sL$ into a subgroup $\sL\subset\sP$.

We are primarily interested in equivariant vector bundles on $X$. It is well known that any irreducible
equivariant vector bundle on $X$ is of the form
$\CU^\lambda\otimes\left(V/\CU\right)^\mu$.
In terms of representations, these are restrictions on $\sP$ of the corresponding irreducible representation
of~$\sL$.

On the other hand, one can restrict any finite-dimensional representation of $\sP$ on $\sL$. As the latter
is reductive, the restricted representation can be decomposed into a direct sum of irreducibles.
In particular, any equivariant vector bundle admits a filtration with irreducible subquotients.

Besides the ordinary derived category $D^b(X)$ we will need the equivariant derived category $D^b_\sG(X)$.
Given a pair of equivariant vector bundles $\CE,\CF\in Coh^\sG(X)$, the space $\Hom(\CE,\CF)$ carries
a natural structure of a $\sG$-representation. Equivariant morphisms between them are nothing but the invariants
\begin{equation}\label{eq:eqhom}
    \Hom_\sG(\CE,\CF) = \Hom(\CE,\CF)^\sG.
\end{equation}
The group $\sG$ is reductive, so the functor of $\sG$-invariants is exact, and
\begin{equation}\label{eq:eqext}
    \Ext^i_\sG(\CE,\CG) = \Ext^i(\CE,\CG)^\sG.
\end{equation}
We immediately conclude that formulas~\eqref{eq:eqhom} and~\eqref{eq:eqext} hold for arbitrary
$E,F\in D^b_\sG(X)$.

The Borel--Bott--Weil theorem is a tool for computing cohomology of line bundles on the flag space
of a semisimple algebraic group. It may be used for computing cohomology of equivariant vector bundles
on Grassmannians as well. All the details can be found in~\cite{Weyman_Book}.

\section{Exceptional bundles}\label{sec:elm}
Let $X=\Gr(k,V)$ be a Grassmannian, which is not a projective space (in other words, assume $1<k<n-1$).
Let us also fix a pair of integers $0 < w < n-k$ and $0 < h < k$, and consider the partial flag variety
$Z=\Fl(k-h,k;V)$ together with projections on $X$ and $X'=\Gr(k-h,V)$.
\[
  \xymatrix{
      & Z \ar[dl]_p \ar[dr]^q & \\
      X & & X'
  }
\]
Denote by $\CU$ and $\CW$ the tautological sub-bundles of ranks $k$ and $k-h$ on $X$ and~$X'$ respectively.
Without further mention, we denote by the same letters their pullbacks on $Z$.

\begin{proposition}\label{pr:elm}
    Let $(\lambda,\mu)\in\Bl_{w,h}$. Then
    $R^ip_*\left((\CU/\CW)^\lambda\otimes(V/\CW)^{-\mu}\right) = 0$ for all $i>0$.
\end{proposition}

\begin{proof}
    Consider the short exact sequence of vector bundles on $Z$
    \[
    0\to \CU/\CW\to V/\CW\to V/\CU\to 0.
    \]
    There is a filtration on $(V/\CW)^{-\mu}$ with subquotients
    $(V/\CU)^{-\nu}\otimes(\CU/\CW)^{-\mu/\nu}$, where $\nu$ runs over the set of subdiagrams $\nu\subseteq\mu$. Thus,
    $(\CU/\CW)^\lambda\otimes(V/\CW)^{-\mu}$ carries a filtration with associated quotients
    \[
    (\CU/\CW)^\lambda\otimes(V/\CU)^{-\nu}\otimes(\CU/\CW)^{-\mu/\nu},
    \]
    where $\nu\subseteq\mu$. Consider one of the quotients and remark that the bundle $(V/\CU)^{-\nu}$ is
    pulled back from $X$. The projection formula implies
    \[
        R^ip_*\left((\CU/\CW)^\lambda\otimes(V/\CU)^{-\nu}\otimes(\CU/\CW)^{-\mu/\nu}\right) =
        R^ip_*\left((\CU/\CW)^\lambda\otimes(\CU/\CW)^{-\mu/\nu}\right)\otimes(V/\CU)^{-\nu}.
    \]
    At the same time, decompose $(\CU/\CW)^{-\mu/\nu}$ into a direct sum
    \[
        (\CU/\CW)^{-\mu/\nu}=\bigoplus_{\alpha\subseteq\mu} m^{\mu}_{\alpha,\nu}(\CU/\CW)^{-\alpha}, 
    \]
    where $m^{\mu}_{\alpha,\nu}$ is the corresponding Littlewood--Richardson coefficient.
    We see that it is enough to prove vanishing $R^\bullet p_*\left((\CU/\CW)^\lambda\otimes(\CU/\CW)^{-\alpha}\right)=0$
    for all $\alpha\subseteq\mu$. In order to do that, consider a single irreducible summand $(\CU/\CW)^\beta\subset (\CU/\CW)^\lambda\otimes(\CU/\CW)^{-\alpha}$.
    According to the Borel--Bott--Weil theorem, there is a nontrivial higher direct image if and only if
    all the elements in the sequence
    \begin{equation}\label{eprbbw}
        (k+\beta_1,\ldots,k-h+1+\beta_h,k-h,\ldots,1)
    \end{equation}
    are distinct, and the sequence itself is not decreasing. Remark that the first $h$ members
    of~\eqref{eprbbw} are distinct and decreasing. The same is true for the last $k-h$ terms of~\eqref{eprbbw}.
    If there is a nontrivial higher direct image, then $k-h+1+\beta_h<1$. From Lemma~\ref{lm:schprod} we have
    inequalities
    $\beta_h\geq -\alpha_1\geq -\mu_1\geq h-k$, which imply $k-h+1+\beta_h\geq 1$.
    Thus, either the sequence~\eqref{eprbbw} is strongly decreasing, or a couple of its terms are equal.
    In both cases there are no higher direct images.
\end{proof}

We are going to define now the bundles of interest.
For any pair of diagrams $(\lambda,\mu)\in\Bl_{w,h}$ put
\begin{equation}\label{elm:def1}
    \CE^{\lambda,\mu}=p_*\left((\CU/\CW)^\lambda\otimes(V/\CW)^{-\mu}\right).
\end{equation}
Remark that our definition depends on the choice of a block. Despite that, we are soon going to show that
the vector bundle $\CE^{\lambda,\mu}$ depends only on $\lambda$ and $\mu$ (what is important is that the
pair $(\lambda,\mu)$ belongs to \emph{some} block). We will need the following lemma.

\begin{lemma}\label{lm:eeq}
    There is a natural equivariant structure on the bundle $\CE^{\lambda,\mu}$, and it belongs to the
    subcategory $\CE^{\lambda,\mu} \in \left<\CU^\beta\otimes (V/\CU)^{-\nu} \mid \beta\subseteq\lambda,\nu\subseteq\mu\right>\subset D^b_\sG(X)$.
\end{lemma}

\begin{proof}
    Equivariance follows immediately from the definition of $\CE^{\lambda,\mu}$ as a direct image
    of an equivariant bundle under the equivariant morphism $p$. It follows from the proof of
    Proposition~\ref{pr:elm} that $\CE^{\lambda,\mu}$ admits a filtration with associated quotients of the form
    \[
        p_*\left((\CU/\CW)^\lambda\otimes(\CU/\CW)^{-\alpha}\right)\otimes (V/\CU)^{-\nu},
    \]
    where $\nu\subseteq\mu$. We have also seen that the only summands
    $(\CU/\CW)^\beta\subset (\CU/\CW)^\lambda\otimes(\CU/\CW)^{-\alpha}$ having a nontrivial direct image
    are those with $\beta_h\geq 0$. At the same time, as a result of Lemma~\ref{lm:schprod},
    $\beta_i\leq\lambda_i$ for all
    $i=1,\ldots,h$ and hence $\beta\subseteq\lambda$.
\end{proof}

\begin{proposition}\label{pr:eqec}
    The bundles $\left<\CU^\gamma\otimes (V/\CU)^{-\delta} \mid (\gamma,\delta)\in\Bl_{w,h}\right>$
    form an exceptional collection in the equivariant derived category $D^b_\sG(X)$.
    In particular, given $(\gamma,\delta),(\alpha,\beta)\in\Bl_{w,h}$
    \[
        \Ext^\bullet_\sG \left(\CU^\gamma\otimes (V/\CU)^{-\delta}, \CU^\alpha\otimes (V/\CU)^{-\beta}\right) \neq 0,
    \]
    if and only if $\gamma\subseteq\alpha$, $\delta\subseteq\beta$.
\end{proposition}

\begin{proof}
    By definition
    \begin{align*}
        \Ext^\bullet_\sG \left(\CU^\gamma\otimes (V/\CU)^{-\delta}, \CU^\alpha\otimes (V/\CU)^{-\beta}\right) &
        = \Ext^\bullet \left(\CU^\gamma\otimes (V/\CU)^{-\delta}, \CU^\alpha\otimes (V/\CU)^{-\beta}\right)^{\sG} = \\
        & = H^\bullet \left(X, \CU^\alpha\otimes\CU^{-\gamma}\otimes (V/\CU)^{-\beta}\otimes (V/\CU)^\delta\right)^{\sG}.
    \end{align*}
    Let $\CU^\lambda$ be an irreducible summand of $\CU^\alpha\otimes\CU^{-\gamma}$,
    $(V/\CU)^\mu$ be an irreducible summand of $(V/\CU)^{-\beta}\otimes (V/\CU)^\delta$.
    
    The space of interest $\Ext^\bullet_\sG$ is nontrivial, if and only if for a pair of summands
    $\CU^\lambda$ and $(V/\CU)^\mu$ the sequence
    \begin{equation}\label{eqortsq}
        (n+\mu_1,\ldots,k+1+\mu_{n-k}, k+\lambda_1,\ldots,1+\lambda_k)
    \end{equation}
    is a permutation of $n$ consecutive integers.
    
    It follows from $(\alpha,\beta)\in\Bl_{w,h}$ that $\alpha_k=0$ and $\beta_{n-k}=0$.
    Applying Lemma~\ref{lm:schprod}, we see that $\lambda_k\leq \alpha_k = 0$, $\mu_1\geq -\beta_{n-k} = 0$.
    This, the sequence~\eqref{eqortsq} can only be a nontrivial permutation of $(1,\ldots,n)$.
    We deduce that a necessary condition would be $\lambda\geq 0$ and $-\mu\geq 0$. The latter is possible
    if and only if $\gamma\subseteq\alpha$, $\delta\subseteq\beta$, which is the semi-orthogonality condition.
    
    Finally, in the case $\gamma=\alpha$, $\delta=\beta$, the only suitable summands are trivial and come with
    multiplicity one. Thereby, the objects are exceptional.
\end{proof}

\begin{remark}
    It is not hard to check that the set of \emph{all} irreducible equivariant vector bundles on $X$
    admits an ordering, such that one gets an infinite full exceptional collection in the equivariant
    derived category $D^b_\sG(X)$ (see~\cite[Theorem~3.4]{KuzPol}).
\end{remark}

\begin{proposition}\label{pr:elmmain}
    Let $(\alpha,\beta),(\lambda,\mu)\in\Bl_{w,h}$ be such that
    $\Ext^\bullet \left(\CU^\alpha\otimes (V/\CU)^{-\beta}, \CE^{\lambda,\mu}\right) \neq 0$.
    Then $\alpha\supseteq\lambda$, and $\alpha=\lambda$ implies $\beta\supseteq\mu$.
\end{proposition}

\begin{proof}
    By adjunction, rewrite
    \begin{align*}
        \Ext^i \left(\CU^\alpha\otimes (V/\CU)^{-\beta}, \CE^{\lambda,\mu}\right)
        & = \Ext^i \left(\CU^\alpha\otimes (V/\CU)^{-\beta}, p_*\left((\CU/\CW)^\lambda\otimes(V/\CW)^{-\mu}\right)\right) = \\
        & = H^i\left(Z, \CU^{-\alpha}\otimes (V/\CU)^{\beta}\otimes (\CU/\CW)^\lambda\otimes(V/\CW)^{-\mu}\right).
    \end{align*}
    Consider the Leray spectral sequence
    \begin{equation}\label{specselm}
      \begin{aligned}
        E^{ab}_2 & = H^a\left(X', R^bq_*\left(\CU^{-\alpha}\otimes (V/\CU)^{\beta}\otimes (\CU/\CW)^\lambda\otimes(V/\CW)^{-\mu}\right)\right) \\
        & \Rightarrow H^{a+b}\left(Z, \CU^{-\alpha}\otimes (V/\CU)^{\beta}\otimes (\CU/\CW)^\lambda\otimes(V/\CW)^{-\mu}\right).
      \end{aligned}
    \end{equation}
    Let us take one term of~\eqref{specselm}. By the projection formula,
    \[
        R^bq_*\left(\CU^{-\alpha}\otimes (V/\CU)^{\beta}\otimes (\CU/\CW)^\lambda\otimes(V/\CW)^{-\mu}\right)
        = R^bq_*\left(\CU^{-\alpha}\otimes (V/\CU)^{\beta}\otimes (\CU/\CW)^\lambda\right)\otimes(V/\CW)^{-\mu}.
    \]
    There is a filtration on $\CU^{-\alpha}$ with associated quotients
    of the form $(\CU/\CW)^{-\tau}\otimes\CW^{-\nu}$, where $\tau\subseteq\alpha$.
    
    Let us study the direct images
    \begin{equation}\label{rqelm}
        R^\bullet q_*\left((\CU/\CW)^{-\tau}\otimes\CW^{-\nu}\otimes (V/\CU)^{\beta}\otimes (\CU/\CW)^\lambda\right)
        = R^\bullet q_*\left((\CU/\CW)^{-\tau}\otimes (\CU/\CW)^\lambda\otimes (V/\CU)^{\beta}\right)\otimes\CW^{-\nu}.
    \end{equation}
    Recall that $Z$ is the relative Grassmannian $\Gr_{X'}(h, V/\CW)$.
    Consider an irreducible summand $(\CU/\CW)^\gamma\subset (\CU/\CW)^{-\tau}\otimes (\CU/\CW)^\lambda$.
    By the Borel--Bott--Weil theorem~\eqref{rqelm} is nontrivial if and only if all the terms in
    \begin{equation}\label{bbwsqelm}
        (n-k+h+\beta_1,\ldots,h+1+\beta_{n-k}, h+\gamma_1,\ldots,1+\gamma_1)
    \end{equation}
    are distinct. As $\beta\in\You_{k-h,n-k-w}$, one has $\beta_{n-k-w+1}=\ldots=\beta_{n-k}=0$.
    Thus, the first $n-k$ terms of~\eqref{bbwsqelm} are of the form
    \[
        (n-k+h+\beta_1,\ldots,h+w+1+\beta_{n-k-w}, h+w,\ldots,h+1).
    \]
    The next term equals $h+\gamma_1$, and by Lemma~\ref{lm:schprod}, $w\geq\lambda_1\geq\gamma_1$.
    Summing up, only the summands $(\CU/\CW)^\gamma$ with $\gamma_1\leq 0$ (which is equivalent to
    $-\gamma\geq 0$) contribute to higher direct images.
    By Lemma~\ref{lm:schprod}, the latter condition may hold only if
    $\lambda\subseteq\tau$. In turn, $\tau\subseteq\alpha$, which implies $\lambda\subseteq\alpha$.
    
    The only remaining case is $\lambda=\tau=\alpha$, the only interesting associated quotient of
    $\CU^{-\alpha}$ being equal to $(\CU/\CW)^{-\alpha}$, the only interesting summand $\gamma=0$,
    and the only non-trivial direct image
    \[
        R^0 q_*\left((\CU/\CW)^{-\tau}\otimes (\CU/\CW)^\lambda\otimes (V/\CU)^{\beta}\right)\otimes\CW^{-\nu} = (V/\CW)^\beta.
    \]
    All the nontrivial terms of the spectral sequence~\eqref{specselm} are concentrated in a single row
    and are equal to
    \[
        E^{a0}_2=H^a\left(X',(V/\CW)^\beta\otimes(V/\CW)^{-\mu}\right) = \Ext^a\left((V/\CW)^\mu, (V/\CW)^\beta\right).
    \]
    Finally, $(V/\CW)^\mu$ and $(V/\CW)^\lambda$ both belong to the dual Kapranov's collection on $X'$,
    thus we get the desired condition $\beta\supseteq\alpha$. Remark that for $\alpha=\lambda$ and $\beta=\mu$
    the only non-trivial space
    \[
        \Ext^\bullet \left(\CU^\alpha\otimes (V/\CU)^{-\beta}, \CE^{\lambda,\mu}\right)=
        \Hom\left(\CU^\alpha\otimes (V/\CU)^{-\beta}, \CE^{\lambda,\mu}\right) = \kk
    \]
    is one-dimensional, and all the morphisms are equivariant.
\end{proof}

We introduce a partial order on the set $\Bl_{w,h}$
\[
    (\alpha,\beta)\prec_1 (\lambda,\mu)\quad \Leftrightarrow\quad \alpha\supset\lambda\text{ or }\alpha=\lambda, \beta\supset\mu.
\]

\begin{corollary}\label{cor:elmort1}
    There is an exceptional collection in $D^b(X)$ of the form $\left<\CE^{\lambda,\mu} \mid (\lambda,\mu)\in\Bl_{w,h}\right>$
    with any total order $<$ refining $\prec_1$.
\end{corollary}

\begin{proof}
    Exceptionality of the bundles $\CE^{\lambda,\mu}$ was already proved in Proposition~\ref{pr:elmmain}.
    It remains to check semi-orthogonality.
    Let $(\alpha,\beta) > (\lambda,\mu)$ be a pair of elements of $\Bl_{w,h}$. Then $(\alpha,\beta) \not\prec_1 (\lambda,\mu)$.
    We wish to show that $\Ext^\bullet(\CE^{\alpha,\beta}, \CE^{\lambda,\mu}) = 0$.
    There is a filtration on $\CE^{\alpha,\beta}$ with associated quotients of the form $\CU^\tau\otimes (V/\CU)^{-\nu}$ with
    $\tau\subseteq\alpha$, $\nu\subseteq\beta$.
    Meanwhile, the condition $(\alpha,\beta)\not\prec_1 (\lambda,\mu)$ is closed under passing to subdiagrams     $\tau\subseteq\alpha$, $\nu\subseteq\beta$. It remains to apply Proposition~\ref{pr:elmmain}.
\end{proof}

\begin{corollary}
    Let $(\lambda,\mu)\in\Bl_{w,h}$. The bundles
    $\left<\CE^{\alpha,\beta} \mid \alpha\subseteq\lambda, \beta\subseteq\mu\right>$
    as objects of the equivariant derived category $D^b_\sG(X)$ form the left dual exceptional collection
    to the collection
    $\left<\CU^\alpha\otimes (V/\CU)^{-\beta} \mid \alpha\subseteq\lambda, \beta\subseteq\mu\right>$.
\end{corollary}

\begin{proof}
    In is enough to check that
    \[
        \Ext^i_\sG\left(\CU^\alpha\otimes(V/\CU)^{-\beta}, \CE^{\tau,\nu}\right) = \begin{cases}
            \kk, & i = 0, \alpha = \tau, \beta = \nu; \\
            0, & \text{otherwise.}
        \end{cases}
    \]
    By definition
    $\Ext^i_\sG\left(\CU^\alpha\otimes(V/\CU)^{-\beta}, \CE^{\tau,\nu}\right)
     = \Ext^i\left(\CU^\alpha\otimes(V/\CU)^{-\beta}, \CE^{\tau,\nu}\right)^\sG$.
    It follows from Proposition~\ref{pr:elmmain} that
        $\Ext^\bullet_\sG\left(\CU^\alpha\otimes(V/\CU)^{-\beta}, \CE^{\tau,\nu}\right) = 0$
    if $\tau\not\subseteq\alpha$ or $\nu\not\subseteq\alpha$.
    
    Let $\tau\subseteq\alpha, \nu\subseteq\alpha$. There is a filtration on $\CE^{\tau,\nu}$
    with associated quotients of the form $\CU^\gamma\otimes(V/\CU)^{-\delta}$ with $\gamma\subseteq\tau$
    and $\delta\subseteq\nu$. We deduce the remaining equalities from Proposition~\ref{pr:eqec}.
\end{proof}

\begin{corollary}
    Assume that the pair $(\lambda,\mu)$ belongs to \emph{some} block.
    Then the vector bundle $\CE^{\lambda,\mu}$ depends only on $\lambda$ and $\mu$,
    and not on the choice of $\Bl_{w,h}\ni (\lambda,\mu)$.
\end{corollary}

\begin{proof}
    The subcategory $\CT=\left<\CU^\alpha\otimes (V/\CU)^{-\beta}\mid \alpha\subseteq\lambda, \beta\subseteq\mu\right>\subseteq D^b_\sG(X)$
    does not depend on the choice of a block. Meanwhile, $\CE^{\lambda,\mu}\in\CT$ and is uniquely defined
    by the dual collection conditions.
\end{proof}

For a given element $(\lambda,\mu)\in\Bl_{w,h}$ one could repeat the entire discussion for the partial flag
variety $\Fl(k,n-w;V)$
\[
  \xymatrix{
      & \Fl(k,n-w;V) \ar[dl]_g \ar[dr]^f & \\
      \Gr(n-w,V) & & \Gr(k,V)
  }
\]
and the bundle $(\CK/\CU)^{-\mu}\otimes\CK^\lambda$, where $\CU\subset\CK$ is the universal
flag on $\Fl(k,n-w;V)$.

\begin{proposition}
    Using the notation introduced earlier,
    \begin{equation}\label{elm:def2}
        \CE^{\lambda,\mu} = f_*\left((\CK/\CU)^{-\mu}\otimes\CK^\lambda\right).
    \end{equation}
\end{proposition}

\begin{proof}
    The invariant description of the resulting direct images as elements of dual exceptional collections
    remains the same.
\end{proof}

The alternative Corollary~\ref{cor:elmort1}, which uses~\eqref{elm:def2} to define our vector bundles,
proves semi-orthogonality with respect to the partial order
\[
    (\alpha,\beta)\prec_2 (\lambda,\mu)\quad \Leftrightarrow\quad \beta\supset\mu\text{ or }\beta=\mu, \alpha\supset\lambda.
\]
Thus, we get the following.

\begin{corollary}\label{cor:block}
    There is an exceptional collection in $D^b(X)$ of the form
    $\left<\CE^{\lambda,\mu}\ \mid\ (\lambda,\mu)\in\Bl_{w,h}\right>$
    with any ordering refining
    $(\alpha,\beta)\preceq (\lambda,\mu)\ \Leftrightarrow\ \beta\supseteq\mu\text{ and }\alpha\supseteq\lambda$.
\end{corollary}

\section{Staircase complexes}
A class of exact complexes of equivariant vector bundles on Grassmannians was constructed in~\cite{Fon1}. These complexes, called staircase, are a natural generalization of the twist by $\CO(1)$ of the Koszul complex on $\PP(V)$
\[
  0\to \CO(1-n) \to \CO(2-n)\otimes V \to \CO(3-n)\otimes\Lambda^2V\to\ldots
  \to\CO(-1)\otimes\Lambda^{n-2}V\to\CO\otimes\Lambda^{n-1}V\to\CO(1)\to 0,
\]
and help express twists of exceptional objects in Kapranov's collection in terms of the collection itself.

\begin{theorem}[{\cite[Proposition~5.3]{Fon1}}]\label{stcthm1}
    Let $\lambda\in\You_{n-k,k}$ be a Young diagram of maximal width, that is $\lambda_1=n-k$. Then there is
    an exact sequence of equivariant vector bundles on $X$ of the form
    \begin{equation}\label{stc1}
            0\to \CU^\lambda \to \CU^{\dcut{\lambda}{1}}\otimes \Lambda^{\ncut{\lambda}{1}}V \to
                \CU^{\dcut{\lambda}{2}}\otimes \Lambda^{\ncut{\lambda}{2}}V \to \ldots
            \to \CU^{\dcut{\lambda}{n-k}}\otimes \Lambda^{\ncut{\lambda}{n-k}}V\to \CU^{\lambda\{1\}}(1)\to 0.
    \end{equation}
\end{theorem}

Complexes of the form~\eqref{stc1}, as well as their duals, will be called \emph{staircase.}
Using the isomorphism $\Gr(k,V)\simeq\Gr(n-k,V^*)$, which identifies the tautological subbundle on
the latter with the quotient bundle $(V/\CU)^*$ on the former, we immediately obtain staircase
complexes on $X$ of the following form.
\begin{theorem}\label{stcthm2}
    Let $\mu\in\You_{k,n-k}$ be a Young diagram of maximal width, that is $\mu_1=k$. Then there is
    an exact sequence of equivariant vector bundles on $X$ of the form
    \begin{equation*}
        \begin{split}
            0\to (V/\CU)^{-\mu} &\to \CU^{-\dcut{\mu}{1}}\otimes \Lambda^{\ncut{\mu}{1}}V^* \to
                (V/\CU)^{-\dcut{\mu}{2}}\otimes \Lambda^{\ncut{\mu}{2}}V^* \to \ldots \\
            \ldots &\to \CU^{-\dcut{\mu}{k}}\otimes \Lambda^{\ncut{\mu}{k}}V^* \to (V/\CU)^{-\mu\{1\}}(1)\to 0.
        \end{split}
    \end{equation*}
\end{theorem}

It turns out one can construct generalized staircase complexes for exceptional bundles $\CE^{\lambda,\mu}$
considered in the previous section by means of an easy geometric argument.
Recall that in terms of partial flag varieties $Z_1=\Fl(k-h,k;V)$ and $Z_2=\Fl(k,n-w;V)$
\[
  \xymatrix{
        & Z_2 \ar[ld]_g \ar[rd]^f &   & Z_1 \ar[ld]_p \ar[rd]^q &    \\
    X'' &                         & X &                         & X' \\
  }
\]
where $X'=\Gr(k-h,V)$ and $X''=\Gr(n-w,V)$, there are explicit formulas
\[
    \CE^{\lambda,\mu} = p_*\left((\CU/\CW)^\lambda\otimes(V/\CW)^{-\mu}\right), \qquad
    \CE^{\lambda,\mu} = f_*\left((\CK/\CU)^{-\mu}\otimes\CK^\lambda\right),
\]
where $\CW$ and $\CK$ are tautological subbundles of ranks $k-h$ and $n-w$ on $Z_1$ and $Z_2$ respectively.

\begin{theorem}\label{thm:stcl}
    For any pair $(\lambda,\mu)\in\Bl_{w,h}$ with $\lambda_1=w$, there exists an exact sequence
    on $X$ of the form
    \begin{equation}\label{eq:stcl}
            0\to\CE^{\lambda,\mu} \to\CE^{\dcut{\lambda}{1},\mu}\otimes\Lambda^{\ncut{\lambda}{1}}V\to
            \CE^{\dcut{\lambda}{2},\mu}\otimes\Lambda^{\ncut{\lambda}{2}}V\to \ldots \\
            \to \CE^{\dcut{\lambda}{w},\mu}\otimes\Lambda^{\ncut{\lambda}{w}}V\to \CE^{\lambda\{1\}, \mu(1)}(1)\to 0.
    \end{equation}
\end{theorem}
\begin{proof}
Let $(\lambda,\mu)\in\Bl_{w,h}$ be such that $\lambda_1=w$. By Theorem~\ref{stcthm1}, there is an exact
sequence on $X''$ of the form
\begin{equation}\label{eq:stcelml}
    0\to \CK^\lambda \to \CK^{\dcut{\lambda}{1}}\otimes \Lambda^{\ncut{\lambda}{1}}V \to
        \CK^{\dcut{\lambda}{2}}\otimes \Lambda^{\ncut{\lambda}{2}}V \to \ldots \\
    \to \CK^{\dcut{\lambda}{w}}\otimes \Lambda^{\ncut{\lambda}{w}}V\to \CK^{\lambda\{1\}}\otimes\CO_{X''}(1)\to 0.
\end{equation}

Apply the functor $f_*\left((\CK/\CU)^{-\mu}\otimes g^*(-)\right)$ to the complex~\eqref{eq:stcelml}
and remark that all the terms, except for the rightmost one, are automatically the same as in~\eqref{eq:stcl}.
Let us deal with the last term.
\[
\begin{split}
    f_*\left((\CK/\CU)^{-\mu}\otimes g^*\left(\CK^{\lambda\{1\}}\otimes\CO_{X''}(1)\right)\right) =
    f_*\left((\CK/\CU)^{-\mu}\otimes \CK^{\lambda\{1\}}\otimes\det \CK^*\right) = \\
    = f_*\left((\CK/\CU)^{-\mu}\otimes \CK^{\lambda\{1\}}\otimes\det \CU^*\otimes\det (\CK/\CU)^*\right) = \\
    = f_*\left((\CK/\CU)^{-\mu}\otimes \CK^{\lambda\{1\}}\otimes\det (\CK/\CU)^*\right)\otimes\CO_{X}(1) = \\
    = f_*\left((\CK/\CU)^{-\mu(1)}\otimes \CK^{\lambda\{1\}}\right)\otimes\CO_{X}(1).
\end{split}
\]
Here the first equality comes by definition, the second comes from the short exact sequence
\[
    0\to (\CK/\CU)^*\to \CK^*\to \CU^*\to 0,
\]
and the third is the projection formula. Remark that the pair $\left(\lambda\{1\},\mu(1)\right)$ belongs
to the block $\Bl_{w,h-1}$. Thus,
\[
    f_*\left((\CK/\CU)^{-\mu(1)}\otimes \CK^{\lambda\{1\}}\right) = \CE^{\lambda\{1\},\mu(1)}.
\]
It remains to recall that the higher direct images vanish.
\end{proof}

If one thinks of $\CE^{\lambda,\mu}$ as the direct image
$p_*\left((\CU/\CW)^\lambda\otimes(V/\CW)^{-\mu}\right)$ and uses staircase complexes from
Theorem~\ref{stcthm2}, an analogous argument shows the following result.

\begin{theorem}\label{thm:stcm}
    For any pair $(\lambda,\mu)\in\Bl_{w,h}$, such that $\mu_1=k-h$, there is an exact sequence on $X$
    of the form
    \begin{equation}\label{eq:stcm}
            0\to\CE^{\lambda,\mu} \to\CE^{\lambda,\dcut{\mu}{1}}\otimes\Lambda^{\ncut{\mu}{1}}V^*\to
            \CE^{\lambda,\dcut{\mu}{2}}\otimes\Lambda^{\ncut{\mu}{2}}V^*\to \ldots
            \to \CE^{\lambda,\dcut{\mu}{k-h},\mu}\otimes\Lambda^{\ncut{\mu}{k-h}}V^*\to \CE^{\lambda(1), \mu\{1\}}(1)\to 0.
    \end{equation}
\end{theorem}

\section{Conjecture of Kuznetsov and Polishchuk}
Before we proceed to the proof of Conjecture~\ref{conj}, let us make a couple of
preliminary remarks. Semi-orthogonal decompositions, which form the core of the statement,
are related to some paths in a rectangle. Despite the continuous nature of the space
of paths, the conjectural set of decompositions is discrete. Recall that the components
depend on intersection points of the path $\Gamma$ with the grid
$N=\RR\times\ZZ\cup\ZZ\times\RR\subset\RR^2$.
Let $p$ be such a point. There are three principal cases: only the abscissa of $p$ is integral,
only the ordinate of $p$ is integral, or both of the coordinates of $p$ are integral.
In each of these cases the corresponding semi-orthogonal component depends only on the index
of $p$ in the ordered sequence of intersection points of $\Gamma$ with $N$ and on the block $\Bl_p$.
The block itself depends only on integral parts of the coordinates of $p$. Thus,
without loss of generality, we assume that all the paths of interest are piecewise-linear and
connect points with half-integer coordinates.

Moreover, we have already managed to prove the part of the conjecture concerning the structure of
the semi-orthogonal components.

\begin{proposition}\label{pr:blec}
    For any $p\in\Pi$, the subcategory $\CB_p\subset D^b(X)$ is generated by an exceptional collection.
\end{proposition}

\begin{proof}
    Anti-autoequivalences transform exceptional collections into exceptional collections with the opposite
    order. Thus, it is sufficient to prove that the subcategory
    \[
        \CB^*_p=\left<\CU^{\lambda}\otimes (V/\CU)^{-\mu}\ \mid\ (\lambda,\mu)\in\Bl_p\right>
    \]
    is generated by a full exceptional collection. According to Corollary~\ref{cor:block},
    for any point $p$ with integral coordinates
    $\CB^*_p=\left<\CE^{\lambda,\mu} \mid (\lambda,\mu)\in\Bl_p\right>$.
    It remains to observe that $\Bl_p=\Bl_{\ceil{x},\floor{y}}\cap\Bl_{\floor{x},\ceil{y}}$.
\end{proof}

We finally turn to the proof of semi-orthogonality and fullness of the decompositions.
It will be convenient to slightly modify the statement of the conjecture.
Let $p=(x,y)$ be a point inside the rectangle $\Pi$ with non-integral coordinates.
Let $\Gamma$ be a strictly monotonous path passing through $p$ and going to $(n-k,k)$.
Denote by
$p_0,\ldots,p_{l(\gamma)}$ the points on $\Gamma$ to the right of $p$,
with at least one integral coordinate. As before, we order $p_i$ with respect to any of the coordinates.
We are going to prove the following statement.

\begin{theorem}\label{thm:kpmain}
    There is a semi-orthogonal decomposition in $D^b(X)$ of the form
    \[
        \left<\CB_{p_0},\CB_{p_1}(1),\ldots,\CB_{p_{l(\Gamma)}}(l(\Gamma))\right>=\CB_{\geq p}^\Gamma\subset D^b(X).
    \]
    Moreover, the subcategory generated by the blocks does not depend on the choice of $\Gamma$ passing
    through $p$. In the following we will denote this subcategory by $\CB_{\geq p}$.
\end{theorem}

\begin{lemma}\label{lm:diagort}
    Let $p=(x,y)$ and $p'=(x+t,y+t)$ be a pair of points with integral coordinates in the rectangle
    $\Pi$, where $t>0$. Then
    \[
        \CB_p\subset \CB_{p'}(t)^\perp.
    \]
\end{lemma}

\begin{proof}
    Denote $w=n-k-x$ and $h=y$. It suffices to check that for all
    \[
    (\lambda,\mu)\in\Bl_{p'}=\You_{w-t,h+t}\times\You_{k-h-t,n-k-w+t}\quad\text{and}\quad(\alpha,\beta)\in\Bl_{p}=\You_{w,h}\times\You_{k-h,n-k-w}
    \]
    one has
    \[
        \Ext^\bullet\left(\CU^{-\lambda}\otimes(V/\CU)^{\mu}(t),\CU^{-\alpha}\otimes(V/\CU)^{\beta}\right)
        = H^\bullet\left(X,\CU^{\lambda}\otimes\CU^{-\alpha}\otimes(V/\CU)^{-\mu}\otimes(V/\CU)^{\beta}(-t)\right) = 0.
    \]
    Let $\CU^\gamma$ and $(V/\CU)^\delta$ be irreducible summands of $\CU^{\lambda}\otimes\CU^{-\alpha}$
    and $(V/\CU)^{-\mu}\otimes(V/\CU)^{\beta}$ respectively. It is enough to prove that
    $H^\bullet\left(X, \CU^\gamma\otimes(V/\CU)^\delta(-t)\right) = 0$.
    By the Borel--Bott--Weil theorem the latter holds if and only if at least two terms in
    \begin{equation}\label{eq:dortsq}
        (n+\delta_1,n-1+\delta_2,\ldots,k+1+\delta_{n-k},k+t+\gamma_1,k-1+t+\gamma_2,\ldots,1+t+\gamma_k)
    \end{equation}
    coincide.
    
    The first $n-k$ and the last $k$ terms of~\eqref{eq:dortsq} form strictly decreasing subsequences.
    Consider the terms with indices $n-k-w+1,\ldots,n-k$. By Lemma~\ref{lm:schprod} and the previous remark,
    \[
        k+w+\delta_{n-k-w+1}\leq k+w\quad\text{and}\quad k+1+\delta_{n-k}\geq k+1-(k-h-t)=h+1+t.
    \]
    Thus, all the terms with given indices belong to $[h+1+t,\ldots,k+w]$.
    On the other hand, consider the terms with indices $n-k+1,\ldots,n-h$. The same arguments imply
    \[
        k+t+\gamma_1\leq k+t+w-t=k+w\quad\text{and}\quad h+1+t+\gamma_{k-h}\geq h+1+t.
    \]
    Thus, these terms belong to $[h+1+t,\ldots,k+w]$. Summarizing, $w+k-h$ terms of
    the sequence~\eqref{eq:dortsq}
    belong to an integral segment of size~$w+k-h-t$. The pigeonhole principle finishes the job.
\end{proof}

\begin{proof}[Proof of Theorem~\ref{thm:kpmain}]
    Without loss of generality one may assume that the point $p$ has coordinates of the form
    $(x-\frac{1}{2}, y-\frac{1}{2})$, where $1\leq x\leq n-k$ and $1\leq y\leq k$ are integers.
    Let us prove the statement by descending induction on the sum of the coordinates of $p$.
    For convenience, denote
    $w=n-k-x$, $h=y$ and $\CB_i=\CB_{p_i}$.
    
    \emph{Induction basis.}
    Let us show that the statement holds for the set of points with $y=k$ or $x=n-k$.
    Consider the case $y=k$. The only option for the path is to intersect vertical segments.
    
    
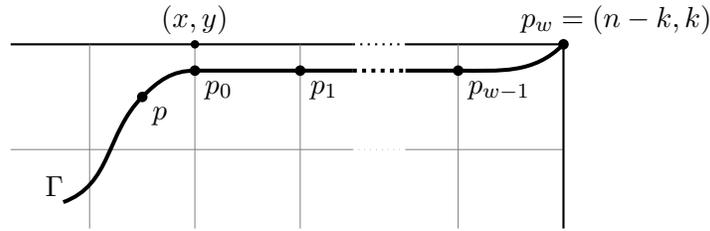
\begin{center}
    \begin{tikzpicture}[scale=0.7]
        \draw [thick] (-1.5,2) -- (5,2);
        \draw [thick, dotted] (5,2) -- (6,2);
        \draw [thick] (6,2) -- (9,2);
        \draw [help lines] (-1.5,0) -- (5,0);
        \draw [dotted, help lines] (5,0) -- (6,0);
        \draw [help lines] (6,0) -- (9,0);
        \foreach \x in {0,2,4,7}
            \draw [help lines] (\x, 2) -- (\x, -1.5);
        \draw [thick] (9,2) -- (9,-1.5);

        \draw [line width=1.5] (-0.5, -1) to [out=20,in=225] (1,1) to [out=45,in=180] (2,1.5) to (5,1.5);
        \draw [line width=1.5, dotted] (5,1.5) to (6,1.5); 
        \draw [line width=1.5] (6,1.5) to (7,1.5) to [out=0,in=225] (9,2);

        \fill [black] (1,1) circle [radius=0.1];
        \node [below right] at (1,1) {$p$};
        \fill [black] (2,2) circle [radius=0.08];
        \node [above] at (2,2) {$(x,y)$};
        \node [above right] at (8,2) {$p_w=(n-k,k)$};
        \node [below right] at (2,1.5) {$p_0$};
        \node [below right] at (4,1.5) {$p_1$};
        \node [below right] at (7,1.5) {$p_{w-1}$};
        \foreach \x in {2,4,7}
            \fill [black] (\x, 1.5) circle [radius=0.1];
        \fill [black] (9,2) circle [radius=0.1];
        \node [left] at (-0.3,-0.7) {$\Gamma$};
    \end{tikzpicture}
    \end{center}
    
    We expect a semi-orthogonal decomposition of the form
    $\CB_{\geq p}=\left<\CB_0,\CB_1(1),\ldots,\CB_w(w)\right>$ with components
    $\CB_i(i) = \left<\CU^{-\lambda}(i) \mid \lambda\in\You_{w-i, k-1}\right>$. It is easy to see that
    \[
        \CB_{\geq p}=\left<\CU^{-\lambda} \mid \lambda\in\You_{w,k}\right>.
    \]
    In remains to observe that the occurring vector bundles form a subcollection of Kapranov's exceptional
    collection, and the block subdivision corresponds to the value of $\lambda_k$. Moreover, one has $\CB_{\geq p}=\CB_{(x,y)}$.
    
    The case $x=n-k$ is treated similarly with the only difference being that one gets a subcollection of the
    dual Kapranov's collection.
    
    \emph{Induction step.}
    Consider a special path $\Gamma$, the one that goes along the diagonal through nodes until the border,
    where it turns to the upper right corner.
    
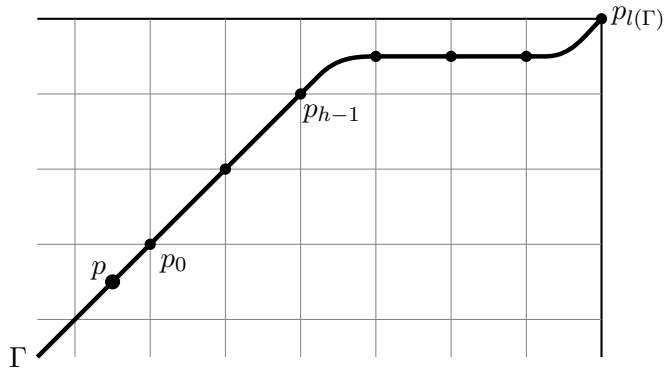
\begin{center}
    \begin{tikzpicture}[scale=0.5]
        \draw [thick] (-1,8) -- (14,8);
        \draw [thick] (14,-1) -- (14,8);
        
        \foreach \x in {0,2,4,6,8,10,12}
            \draw [help lines] (\x, -1) -- (\x, 8);
        \foreach \y in {0,2,4,6}
            \draw [help lines] (-1,\y) -- (14,\y);

        \draw [line width=1.7] (-1,-1) to (6.5,6.5) to [out=45,in=180] (8,7) to (12.5,7) to [out=0,in=225] (14,8);
        \fill [black] (1,1) circle [radius=0.2];
        \node [above left] at (1.1,0.8) {$p$};
        
        \foreach \x in {2, 4, 6}
            \fill [black] (\x,\x) circle [radius=0.15];
        \foreach \x in {8, 10, 12}
            \fill [black] (\x,7) circle [radius=0.15];
        \fill [black] (14,8) circle [radius=0.15];

        \node [below right] at (2,2) {$p_0$};
        \node [below right] at (5.8,6) {$p_{h-1}$};
        \node [right] at (14,8) {$p_{l(\Gamma)}$};

        \node [left] at (-1,-1) {$\Gamma$};
    \end{tikzpicture}
    \end{center}
    By Lemma~\ref{lm:diagort} we immediately get a semi-orthogonal decomposition
    \[
        \CB^\Gamma_{\geq p} = \left<\CB_0,\CB_1(1),\ldots,\CB_{l(\Gamma)}(l(\Gamma))\right>.
    \]
    Let $\Gamma'$ be another path. There are three different options for the point $p_0$: it either
    coincides with a node, lies on a vertical segment, or lies on a horizontal segment. Let us
    treat the cases separately.
    
    Assume $p_0$ coincides with a node, that is $p_0=(x,y)$. Without loss of generality we may assume that
    $\Gamma'$ passes through the point $p'=(x+\frac{1}{2},y+\frac{1}{2})$. Then
    \[
        \CB^\Gamma_{\geq p} = \left<\CB_0,\CB_1(1),\ldots,\CB_{l(\Gamma)}(l(\Gamma))\right> =
         \left<\CB_0,\CB^\Gamma_{\geq p'}(1)\right> = \left<\CB_0,\CB^{\Gamma'}_{\geq p'}(1)\right>,
    \]
    where the first equality holds by definition, and the second from the induction hypothesis.
    Remark that the points of $\Gamma'$ intersecting the grid and situated to the right of $p'$ are
    $p_1,\ldots,p_{l(\Gamma')}$. Expanding $\CB^{\Gamma'}_{\geq p'}(1)$, we get the desired decomposition.
    The subcategory does not change.
    
    Consider the case when $p_0$ lies on a horizontal segment of the grid (one may assume that $p_0=(x-\frac{1}{2},y)$).
    Without loss of generality one may take $\Gamma'$ passing through the point $p'$ with coordinates
    $(x-\frac{1}{2},y+\frac{1}{2})$. By inductive hypothesis $\CB^{\Gamma'}_{p'}$
    does not depend on the part of the path $\Gamma'$ that is to the right of $p'$. On the other hand, $\Bl_{p_0}\subset\Bl_{(x-1,y)}$. From Lemma~\ref{lm:diagort} we get a semi-orthogonal decomposition
    \[
        \CB^{\Gamma'}_{\geq p} = \left<\CB_{p_0},\CB^{\Gamma'}_{\geq p'}(1)\right> = 
        \left<\CB_{p_0}, \CB_{p_1}(1),\ldots, \CB_{p_{l(\Gamma')}}(l(\Gamma'))\right>.
    \]
    
    It remains to show that $\CB^{\Gamma'}_{\geq p}=\CB^{\Gamma}_{\geq p}$.
    By induction hypothesis one may assume that $\Gamma'$ passes through the node of the grid $p_1=(x,y+1)$.
    There is a semi-orthogonal decomposition
    \[
        \CB^{\Gamma'}_{\geq p}=\left<\CB_0,\CB_1(1), \CA\right>,
    \]
    where $\CB_0=\CB_{p_0}$, $\CB_1=\CB_{p_1}$, and the subcategory $\CA$ is generated by the remaining blocks.
    Consider the path $\Gamma''$, passing through the points $p'_0=(x,y)$ and $p'_1=(x+\frac{1}{2}, y+1)$.
    From what we have proven so far,
    \[
        \CB^{\Gamma}_{\geq p}=\CB^{\Gamma''}_{\geq p} = \left<\CB'_0, \CB'_1(1), \CA\right>,
    \]
    where $\CB'_0=\CB_{p'_0}$, $\CB'_1=\CB_{p'_1}$.
    Thus, it is enough to check the equality $\left<\CB_0, \CB_1(1)\right>=\left<\CB'_0,\CB'_1(1)\right>$.
    
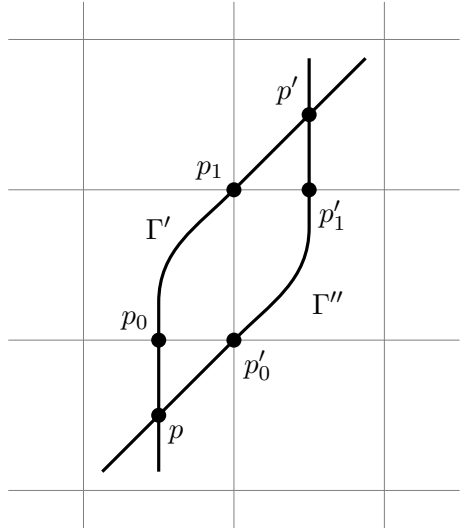
    \begin{center}
    \begin{tikzpicture}[scale=0.5]
        \foreach \x in {0,4,8}
            \draw [help lines] (\x, -1) -- (\x, 13);
        \foreach \y in {0,4,8,12}
            \draw [help lines] (-2,\y) -- (10,\y);

        \pgfmathsetmacro{\rad}{0.2}
        
        \draw [line width=1.2] (2,0.5) to (2,5) to [out=90,in=225] (4,8) to (7.5,11.5);
        \draw [line width=1.2] (0.5,0.5) to (4,4) to [out=45,in=-90] (6,7) to (6,11.5);

        \node at (2,7) {$\Gamma'$};
        \node at (6.5,5) {$\Gamma''$};

        \fill [black] (2,2) circle [radius=\rad];
        \node [below right] at (2,2) {$p$};
        \fill [black] (2,4) circle [radius=\rad];
        \node [above left] at (2,4) {$p_0$};
        \fill [black] (4,8) circle [radius=\rad];
        \node [above left] at (4,8) {$p_1$};
        \fill [black] (6,10) circle [radius=\rad];
        \node [above left] at (6,10) {$p'$};
        \fill [black] (6,8) circle [radius=\rad];
        \node [below right] at (6,8) {$p'_1$};
        \fill [black] (4,4) circle [radius=\rad];
        \node [below right] at (4,4) {$p'_0$};
    \end{tikzpicture}
    \end{center}
    
    According to Proposition~\ref{pr:blec}, for every point $p$ the block $\CB_{p}$
    is generated by an exceptional collection
    $\CB_{p} = \left<\CF^{\lambda,\mu}\mid (\lambda,\mu)\in\Bl_{p}\right>$,
    where $\CF^{\lambda,\mu}=\left(\CE^{\lambda,\mu}\right)^*$.
    As $\Bl_{p_0}\subseteq\Bl_{p'_0}$ and $\Bl_{p'_1}\subseteq\Bl_{p_1}$,
    it remains to check that
    \[
        \CB_1(1)\subset\left<\CB'_0, \CB'_1(1)\right>\quad\text{and}\quad\CB'_0\subset\left<\CB_0, \CB_1(1)\right>.
    \]
    The difference $\Bl_{p_1}\setminus\Bl_{p'_1}$ consists of those pairs $(\lambda,\mu)$,
    such that the first row of $\lambda$ is of maximal length. It is enough to show that for such
    a pair, $\CF^{\lambda,\mu}(1)\in \left<\CB'_0, \CB'_1(1)\right>$.
    The complex dual to~\eqref{eq:stcl} and twisted by $\CO(1)$, is of the form
    \[
            0\to\CF^{\lambda\{1\},\mu(1)} \to\CF^{\dcut{\lambda}{w},\mu}(1)\otimes\Lambda^{\ncut{\lambda}{w}}V^*\to\ldots\to
            \CF^{\dcut{\lambda}{1},\mu}(1)\otimes\Lambda^{\ncut{\lambda}{1}}V^*\to\CF^{\lambda,\mu}(1)\to 0.
    \]
    Finally, remark that the leftmost term belongs to $\CB'_0$,
    while all the intermediate terms belong to $\CB'_1(1)$.
    Thus, $\CB_1(1)\subset\left<\CB'_0, \CB'_1(1)\right>$.
    Considering the same complexes as resolutions of their leftmost terms, we deduce that
    $\CB'_0\subset\left<\CB_0, \CB_1(1)\right>$.
    
    An analogous statement, with the difference that one should use complexes of the form~\eqref{eq:stcm} works in the case when
    $p_0$ lies on a~vertical segment.
\end{proof}

\begin{proof}[Proof of Conjecture~\ref{conj}]
    Without loss of generality we assume that any path of interest $\Gamma$, going from the point $(0,0)$
    to the point $(n-k,k)$ passes through $p=(\frac{1}{2},\frac{1}{2})$.
    By Theorem~\ref{thm:kpmain} and Lemma~\ref{lm:diagort} there is a semi-orthogonal decomposition
    \[
        \left<\CB_{p_0},\CB_{p_1}(1),\ldots,\CB_{p_{l(\Gamma)}}(l(\Gamma))\right>=\left<\CB_{p_0},\CB^{\Gamma}_{\geq p}(1)\right>,
    \]
    and the subcategory generated by its components does not depend on the choice of the path. Let us denote
    it by $\CC$.
    
    According to Proposition~\ref{pr:blec}, every block $\CB_{p_i}(i) = \left<\CF^{\lambda,\mu}(i)\mid (\lambda,\mu)\in\Bl_{p_i}\right>$
    is generated by an exceptional collection,
    where $\CF^{\lambda,\mu}=\left(\CE^{\lambda,\mu}\right)^*$.
    It remains to check that $\CC$ coincides with $D^b(X)$.
    
    Let us choose $\Gamma$ to be a piecewise-linear path, consisting of two segments: the first one connecting
    the points $(0,0)$ and $q=(\frac{1}{2},k-\frac{1}{2})$, the second connecting $q$ and $(n-k,k)$.
    There is a semi-orthogonal decomposition
    \[
        \CC=\left<\CB_{p_0},\CB_{p_1}(1),\ldots,\CB_{p_{n-1}}(n-1)\right> = \left<\CB_{p_0},\ldots,\CB_{p_{k-1}}(k-1),\CB^\Gamma_{\geq q}(k)\right>.
    \]
    When establishing the induction basis during the course of the proof of Theorem~\ref{thm:kpmain},
    we have shown that $\CB^\Gamma_{\geq q}=\CB_{(1,k)}$.
    Thus,
    \[
        \CC=\left<\CB_0,\ldots,\CB_{k-1}(k-1),\CB_k(k)\right>,
    \]
    where $\CB_i=\CB_{p_i}$ for $i=0,\ldots k-1$ and $\CB_k=\CB_{(1,k)}$.
    Moreover, the block $\CB_i$ is generated by vector bundles
    \[
        \CB_i=\left<\CU^{-\lambda}\mid \lambda\in\You_{n-k-1,i}\right>.
    \]
    Consider Kapranov's collection $D^b(X) = \left<\CU^{-\lambda}(k)\mid \lambda\in\You_{n-k,k}\right>$
    twisted by $\CO(k)$.
    It suffices to check that all its elements belong to $\CC$.
    By induction on $t$ we will show that $\CU^{-\lambda}(t)\in\CC$ for all $\lambda\in\You_{n-k,t}$.
    
    The basis case, $t=0$, is obvious.
    Assume that the statement holds for all $s<t$ and let $\lambda\in\You_{n-k,t}$. If $\lambda_1<n-k$,
    then $\lambda\in\You_{n-k-1,t}=\Bl_{p_t}$, and $\CU^{-\lambda}(t)\in\Bl_t(t)\subset\CC$.
    Now, assume $\lambda_1=n-k$.
    Consider the complex~\eqref{stc1}, dualized and twisted by $\CO(t)$. It is of the form
    \[
        0\to\CU^{-\lambda\{1\}}(t-1)\to\CU^{-\dcut{\lambda}{n-k}}(t)\otimes\Lambda^{\ncut{\lambda}{n-k}}V^*\to
        \ldots\to \CU^{-\dcut{\lambda}{1}}(t)\otimes\Lambda^{\ncut{\lambda}{1}}V^*\to\CU^{-\lambda}(t)\to 0.
    \]
    We have just checked that all the intermediate terms belong to $\CC$ (by definition of band cuts,
    diagrams $\dcut{\lambda}{i}$ belong to $\You_{n-k-1,t}$).
    Finally, the first term of the complex belongs to $\CC$ by the induction hypothesis, thus the last as well.
\end{proof}


\def\cprime{$'$}



\end{document}